\documentclass[12pt,letterpaper]{article}

\usepackage{fullpage}
\usepackage{amssymb,amsfonts,amscd,amsthm}
\usepackage[all,arc]{xy}
\usepackage{enumerate, bbm}
\usepackage{hyperref}
\usepackage{mathrsfs}
\usepackage{tikz}
\usepackage{mathtools}
\usepackage{color}
\usepackage{graphicx}
\usepackage{enumerate}
\graphicspath{ {TexImages/} }

\newtheorem{theorem}{Theorem} 
\newtheorem{cor}[theorem]{Corollary}
\newtheorem{obs}[theorem]{Observation}

\newtheorem{claim}[theorem]{Claim}

\newtheorem{conjecture}[theorem]{Conjecture}\theoremstyle{remark}

\newcommand{\C}{\mathcal{C}}
\newcommand{\N}{\mathbb{N}}

\newcommand{\mH}{\mathcal{H}}

\newcommand{\E}{\mathcal{E}}

\newcommand{\ex}{\mathrm{ex}}

\begin{document}
	
	\title{On the Constructor-Blocker Game}
	\author{
		J\'ozsef Balogh\thanks{Department of Mathematics, University of Illinois Urbana-Champaign, Urbana, Illinois 61801, USA. E-mail: \texttt{\{jobal, cechen4\}@illinois.edu}. Balogh was supported in part by NSF grants DMS-1764123, RTG DMS-1937241, FRG DMS-2152488, the Arnold O. Beckman Research Award (UIUC Campus Research Board RB 22000) and the Langan Scholar Fund (UIUC). Chen was supported by FRG DMS-2152488 and UIUC Campus Research Board RB 22000.}
		\and 
		Ce Chen\footnotemark[1]
		\and 
		Sean English\thanks{Department of Mathematics and Statistics, University of North Carolina Wilmington, Wilmington, North Carolina 28403, USA. E-mail: \texttt{EnglishS@uncw.edu}. The work was partially done while S. English was a postdoc at Department of Mathematics, University of Illinois Urbana-Champaign.}
	}
	\date{}
	\maketitle
	
	\begin{abstract}
		In the Constructor-Blocker game, two players, Constructor and Blocker, alternatively claim unclaimed edges of the complete graph  $K_n$. For given graphs $F$ and $H$, Constructor can only claim edges that leave her graph $F$-free, while Blocker has no res\-trictions. Constructor's goal is to build as many copies of $H$ as she can, while Blocker attempts to stop this. The game ends once there are no more edges that Constructor can claim. The score $g(n,H,F)$ of the game is the number of copies of $H$ in Constructor's graph at the end of the game when both players play optimally and Constructor plays first.
		
		In this paper, we extend results of Patk\'os, Stojakovi\'c and Vizer on $g(n, H, F)$ to many pairs of $H$ and $F$: We determine $g(n, H, F)$ when $H=K_r$ and $\chi(F)>r$, also when both $H$ and $F$ are odd cycles, using Szemer\'edi's Regularity Lemma. We also obtain bounds of $g(n, H, F)$ when $H=K_3$ and $F=K_{2,2}$.
	\end{abstract}

	\section{Introduction}
	
	Tur\'an type problems are among the classical problems in extremal graph theory: Given a graph $F$, what is the maximum number of edges in an $n$-vertex $F$-free graph? The \emph{Tur\'an number}, denoted $\ex(n, F)$ is this maximum value. Following Mantel's theorem~\cite{mantel} on $\ex(n,K_3)$, Tur\'an~\cite{Turan} showed that for every $r\geq 2$, the complete balanced $r$-partite graph on $n$ vertices, denoted by $T(n,r)$, is the unique $n$-vertex $K_{r+1}$-free graph with the maximum number of edges. Let $t(n,r):=e(T(n,r))=\ex(n, K_{r+1})$. 
	Erd\H os and Stone~\cite{ES} generalized Tur\'an's theorem and determined $\ex(n, F)$ when $F=K_{r+1}(t)$ is the complete multipartite graph where the partite sets all have size $t$. Erd\H os and Simonovits~\cite{ESS} observed that every graph $F$ with chromatic number $\chi(F)=r+1\geq 3$ is contained in $K_{r+1}(t)$ for some $t$, and thus they determined $\ex(n,F)$ asymptotically for every non-bipartite graph $F$. More recently, there has been a lot of attention given to the \emph{generalized Tur\'an number}: Given two graphs $F$ and $H$, define $\ex(n, H, F)$ to be the maximum number of copies of $H$ in an $n$-vertex $F$-free graph. Note that $\ex(n, F)=\ex(n, K_2, F)$.
	
	Recently, Patk\'os, Stojakovi\'c and Vizer~\cite{PSV} introduced a game version of the generalized Tur\'an number:
	Given two graphs $F$ and $H$, two players, Constructor and Blocker, alternatively claim unclaimed edges of $K_n$. Constructor needs to guarantee that her graph is $F$-free, while Blocker can claim any unclaimed edge. The game ends when all edges of $K_n$ are claimed or Constructor cannot claim any more edges.
	The score of the game is defined to be the number of copies of $H$ in Constructor's graph.
	Constructor aims to maximize the score, while Blocker wants to minimize it.
	Denote by $g(n,H,F)$ the score when both players play optimally and Constructor starts the game.
	We have the following easy observation. 
	\begin{obs}\label{tri-obs}
		\[
		g(n,H,F)\leq \ex(n,H,F).
		\]
	\end{obs}
	Patk\'os, Stojakovi\'c and Vizer~\cite{PSV} determined the exact value of $g(n,H,F)$ when both $F$ and $H$ are stars and when $F=P_4, H=P_3$, where $P_n$ denotes the path on $n$ vertices. They also determined the asymptotics of $g(n,H, F)$ when $F$ is a star and $H$ is a tree, and when $F=P_5, H=K_3$.
	
	Before the study of the Constructor-Blocker game, Hefetz, Krivelevich and Szab\'o~\cite{HKS} introduced the so-called \emph{Avoider-Enforcer game}: Avoider attempts to prevent her graph from having some property $\mathcal{P}$ for as many rounds as possible, while Enforcer tries to force Avoider's graph to have property $\mathcal{P}$ as early as possible. Balogh and Martin~\cite{BM} studied a special case of the Avoider-Enforcer game, where they essentially determined $g(n, K_2, F)$ for every non-bipartite graph $F$.
	
	\begin{theorem}[Theorem 2 in \cite{BM}]\label{jozsi}
		If $\chi(F)=s\geq 3$ and $n$ is large enough, then 
		\[
		\left(\frac{s-2}{s-1}\right)\frac{n^2}{4}-O(s)=\left\lfloor\frac{t(n,s-1)}{2}\right\rfloor\leq g(n, K_2, F)\leq \left(\frac{s-2}{s-1}\right)\frac{n^2}{4}+o(n^2).
		\] 
		Additionally, the last inequality also holds when $s=2$.
	\end{theorem}

	In this paper, we explore $g(n, H, F)$ when $F$ and $H$ are other specific types of graphs. First, we generalize Theorem~\ref{jozsi} and obtain the asymptotics of $g(n, K_r, F)$ when $\chi(F)=s\geq 3$.
	
	\begin{theorem}\label{clique-other}
		If $\chi(F)=s>r\geq 2$, and $n$ is large enough, then
		\[
		g(n, K_r, F)= \left(\frac{1}{2}\right)^{\binom{r}{2}}\binom{s-1}{r}\left(\frac{n}{s-1}\right)^r+o(n^r).
		\]
	\end{theorem}
	
	Taking $r=2$, Theorem~\ref{clique-other} implies that 
	\[
	g(n, K_2, F)= \left(\frac{s-2}{s-1}\right)\frac{n^2}{4}+o(n^2),
	\]
	which coincides with Theorem~\ref{jozsi}. This is not surprising as the proof methods are similar. 
	
	When $\chi(F)=\chi(H)$, we focus on the case where both $F$ and $H$ are odd cycles. Let $\C_{\leq 2k}\coloneqq\{C_4, C_6, \ldots, C_{2k}\}$ be the collection of even cycles with length at most $2k$. If one of $F$ or $H$ is a triangle, we have the following.
	
	\begin{theorem}\label{k3 and odd}
		Let $k\geq 2$. Then,
		\begin{itemize}
			\item[(i)]
			\[
			\Omega\left(\ex(n,\C_{\leq 2k})\right)=g(n,K_3,C_{2k+1})=O\left(\ex(n,C_{2k})\right),
			\]
			and
			\item[(ii)]
			\[
			g(n,C_{2k+1},K_3)=\left(1+o(1)\right)\left(\frac{n}{2(2k+1)}\right)^{2k+1}.
			\]
		\end{itemize}
	\end{theorem}
	
	For $k=2$, it was proved by K\H ov\'ari, S\'os and Tur\'an~\cite{KST} (upper bound), and Erd\H os, R\'enyi and S\'os~\cite{ERS} (lower bound) that \begin{equation}\label{turan-c4}
		\ex(n, C_4)=\left(\frac{1}{2}+o(1)\right)n^{3/2}.
	\end{equation}
	Combinining~\eqref{turan-c4} with Theorem~\ref{k3 and odd} (i), we determine the order of magnitude of $g(n, K_3, C_5)$.
	\begin{cor}
		\[
		g(n,K_3,C_5)=\Theta(n^{3/2}).
		\]
	\end{cor}
	In general, Bondy and Simonovits~\cite{BS} proved that $\ex(n, C_{2k})=O(n^{1+1/k})$ for every $k\geq 2$. It is known that this upper bound gives the correct order of magnitude for $k=2,3$ and $5$, however no matching lower bound is known for any other values of $k$.
	
	When $H\neq K_3$ is an odd cycle shorter than $F=C_{2k+1}$, we determine the order of magnitude of $g(n,H,C_{2k+1})$.
	
	\begin{theorem}\label{odd-odd}
		Let $k>\ell\geq 2$. Then,
		\[
		g(n,C_{2\ell+1}, C_{2k+1})=\Theta(n^\ell).
		\]
	\end{theorem}
	
	For $\chi(F)<\chi(H)$, we study the basic case when $F$ is the small bipartite graph $K_{2,2}=C_4$ and $H$ is a triangle.
	
	\begin{theorem}\label{k3-k22}
		There exists some constant $c>0$ such that 
		\[\Omega(n^{3/2}e^{-c\sqrt{\log n}})=g(n,K_3, C_4)=O(n^{3/2}).\]
	\end{theorem}
	
	We conjecture that the upper bound above gives the correct asymptotic approximation for $g(n, K_3, C_4)$.
	
	\begin{conjecture}
		\[g(n,K_3, C_4)=\Theta(n^{3/2}).\]
	\end{conjecture}
	
	The rest of the paper is organized as follows. In Section~\ref{prelim}, we list some known bounds of $\ex(n,H,F)$ and other results needed for the proofs of our main theorems.
	Section~\ref{chi:::f>h} contains the proof of Theorem~\ref{clique-other}; Section~\ref{chi:::f=h} contains the proofs of Theorems~\ref{k3 and odd} and~\ref{odd-odd}; Section~\ref{chi:::f<h} contains the proof of Theorem~\ref{k3-k22}. Floors and ceilings are omitted if they do not affect the computation much.

	\section{Preliminaries}\label{prelim}
	
	\subsection{Results of generalized Tur\'an number}
	
	Zykov~\cite{zykov}, and independently Erd\H os~\cite{ErdosCliques} determined the value of $\ex(n, H, F)$ when both $F$ and $H$ are cliques. 
	
	\begin{theorem}[Zykov~\cite{zykov}, Erd\H os~\cite{ErdosCliques}]\label{erdoscliques}
		For $s>r$, 
		\[
		\ex(n, K_r, K_s)=\sum_{0\leq i_1\leq \ldots\leq i_r\leq s-2}\prod_{j=1}^r\left\lfloor\frac{n+i_j}{s-1}\right\rfloor.
		\]
	\end{theorem}
	
	It is worth noting that the above bound is realized by the Tur\'an graph $T(n,s-1)$. Alon and Shikhelman~\cite{AS} initiated the systematic study of generalized Tur\'an numbers. Since then, the subject has received considerable attention. We list some of the results below.
	
	\begin{theorem}[Proposition 2.1 in~\cite{AS}]\label{as-chi}
		\[
		\ex(n, H, F)=\Omega(n^{v(H)})
		\] 
		if and only if $F$ is not a subgraph of a blow-up of $H$. Otherwise, $\ex(n, H, F)\leq n^{v(H)-\epsilon}$ for some $\epsilon=\epsilon(H,F)>0$.
	\end{theorem}
	
	\noindent Theorems~\ref{erdoscliques} and~\ref{as-chi} imply the following.
	
	\begin{theorem}[Proposition 2.2 in \cite{AS}]
		If $\chi(F)=s>r$, then 
		\[
		\ex(n, K_r, F)=(1+o(1))\binom{s-1}{r}\left(\frac{n}{s-1}\right)^r.
		\]
	\end{theorem}
	
	Extending the result of Bollob\'as and Gy\H ori~\cite{BG} on $\ex(n, K_3, C_5)$, and the result of Gy\H ori and Li~\cite{GL} on $\ex(n, K_3, C_{2k+1})$ for $k\geq 2$, Alon and Shikhelman~\cite{AS} took a more in-depth look at the case where $H$ is a triangle and $F$ is a longer odd cycle.
	
	\begin{theorem}[Proposition 1.2 in \cite{AS}]\label{thm:::GTtri-cycle} The following upper bounds hold.
		\begin{itemize}
			\item[(i)]
			\[
			\ex(n, K_3, C_5)\leq (1+o(1))\frac{\sqrt{3}}{2}n^{3/2}.
			\]
			\item[(ii)]
			For every $k\geq 2$, 
			\[
			\ex(n,K_3,C_{2k+1})\leq \frac{16(k-1)}{3}\ex(\lceil n/2\rceil,C_{2k}).
			\]
		\end{itemize}
	\end{theorem}
	
	We remark that the upper bound in (i) is sharp up to the multiplicative constant. Let $G_0=(A,B)$ be a $C_4$-free balanced bipartite graph on $2n/3$ vertices with $\Theta(n^{3/2})$ edges. We construct $G$ from $G_0$ by replacing each vertex $a\in A$ with a new edge $x_ay_a$, and joining both new vertices $x_a,y_a$ to all vertices in $N_{G_0}(a)$. Denote $A'=\{x_a, y_a: a\in A\}$, then $G[A']$ is a matching of size $n/3$. Since $G_0$ is bipartite and $C_4$-free, we have that $G$ does not contain $C_5$ as a subgraph. Note that each new edge $x_ay_a$ together with each vertex in $N_{G_0}(a)$ induces a triangle in $G$. Therefore, the number of triangles in $G$ is at least $e(G_0)=\Theta(n^{3/2})$.
	
	Gishboliner and Shapira~\cite{LiorAsaf} generalized Theorem~\ref{thm:::GTtri-cycle} where $H=K_3$ to the case where $H\neq K_3$ is an arbitrary odd cycle shorter than $F$. 
	
	\begin{theorem}[Theorem 1 in \cite{LiorAsaf}]\label{ref-oo}
		If $k>\ell\geq2$, then 
		\[
		\ex(n,C_{2\ell+1}, C_{2k+1})=\Theta(n^\ell).\]
	\end{theorem}
	
	For $\chi(F)<\chi(H)$, Alon and Shikhelman~\cite{AS} investigated the case when $F$ is bipartite and $H$ is a clique.
	
	\begin{theorem}[Theorem 1.3 and Lemma 4.4 in~\cite{AS}]\label{ref-k3-c4}
		Let $r\geq 2$ and $t\geq s$ satisfy $s\geq 2r-2$ and $t\geq (s-1)!+1$. Then, 
		\[
		\ex(n, K_r, K_{s,t})=\Theta\left(n^{r-\binom{r}{2}/s}\right).
		\] 
		If $r=3$, then for every fixed $s\geq 2$ and $t\geq (s-1)!+1$, 
		\[
		\ex(n, K_3, K_{s,t})=\Theta\left(n^{3-3/s}\right).
		\]
	\end{theorem}
	In particular, this includes the case when $r=3$ and $s=t=2$, relevant to Theorem~\ref{k3-k22}: 
	\[
	\ex(n, K_3, C_4)=\Theta(n^{3/2}).
	\]

	\subsection{The game of JumbleG and Szemer\'edi's Regularity Lemma}\label{sec:::jumble}
	
	In the proofs of Theorems~\ref{clique-other} and~\ref{k3 and odd}, we will borrow some of the ideas of Balogh and Martin~\cite{BM}. We first introduce the so-called \emph{game of JumbleG($\epsilon$)}, which is a traditional Maker-Breaker game played on $K_n$, where Maker and Breaker alternate claiming edges of $K_n$ until all edges have been claimed, and the goal of Maker is to end the game with their edges inducing a pseudorandom graph.
	
	Given a graph $G$ and a pair of disjoint vertex sets $S,T \subseteq V(G)$, denote $E_G(S,T)$ the set of edges having one endpoint in $S$ and the other in $T$. Let $e_G(S,T)\coloneqq|E_G(S,T)|$. The \emph{density} of $(S,T)$ in $G$ is $d(S,T)\coloneqq \frac{e_G(S,T)}{|S||T|}$. The pair $(S,T)$ is \emph{$\epsilon$-unbiased} if $|d(S,T)-\frac{1}{2}|\leq \epsilon$. An $n$-vertex graph $G$ is \emph{$\epsilon$-regular} if its minimum degree $\delta(G)\geq (\frac{1}{2}-\epsilon)n$ and every pair $(S,T)$ of disjoint vertex sets $S, T\subseteq V(G)$ with $|S|, |T|>\epsilon n$ is $\epsilon$-unbiased. The game of JumbleG($\epsilon$) is \emph{won by Maker} if Maker's graph ends up $\epsilon$-regular at the end of the game.
	
	\begin{theorem}[Frieze, Krivelevich, Pikhurko, and Szab\'o~\cite{jumbleg}]\label{theorem jumbleg}
		If $\epsilon\geq 2(\log n/n)^{1/3}$ and $n$ is large enough, then Maker has a winning strategy in the game of JumbleG($\epsilon$).
	\end{theorem}
	
	In our applications of Theorem~\ref{theorem jumbleg}, we will often want to play on the complete bipartite graph $K_{n/2,n/2}$ instead of $K_n$. Let the partite sets of $K_{n/2,n/2}$ be denoted by $A$ and $B$. We will call a subgraph $G\subseteq K_{n/2,n/2}$ \emph{$\epsilon$-bipartite-regular} if every pair $S,T$ with $S\subseteq A$, $T\subseteq B$, $|S|,|T|>\epsilon n$ is $\epsilon$-unbiased. 
	
	Fix $K_n$ and let $A,B\subseteq V(K_n)$ with $|A|=|B|=n/2$, and $A\cap B=\emptyset$. Let \emph{Bipartite-JumbleG($\epsilon$)} denote the Maker-Breaker game on $K_n$, where Maker can only claim edges with one endpoint in $A$ and one in $B$, while Breaker can claim any edge. Maker's goal is to create a graph $G\subseteq K_n$ such that $G$ contains an $\epsilon$-bipartite-regular subgraph. Note that if there is a winning strategy for Maker in the JumbleG($\epsilon$) game, then there is also one for the Bipartite-JumbleG($\epsilon$) game. Indeed, Maker can follow a JumbleG($\epsilon$) strategy, and whenever the JumbleG($\epsilon$) strategy expects Maker to make a move that is not allowed in Bipartite-JumbleG($\epsilon$), Maker simply chooses an arbitrary unclaimed edge from $A$ to $B$ instead. Doing so results in Maker's graph having all the edges from $A$ to $B$ that it would have had at the end of JumbleG($\epsilon$) (and then some), giving us the desired subgraph. We state this as a formal observation below, using Theorem~\ref{theorem jumbleg} to give us bounds on when Maker can win Bipartite-JumbleG($\epsilon$).
	
	\begin{obs}\label{obs:::steal}
		If $\epsilon\geq 2(\log n/n)^{1/3}$ and $n$ is large enough, then Maker has a winning strategy in the game of Bipartite-JumbleG($\epsilon$).
	\end{obs}
	
	We also need the following form of Szemer\'edi's Regularity Lemma~\cite{szemeredi}, which is a simplified version of a result of Alon, Fischer, Krivelevich and Szegedy~\cite{AFKS}. As standard, a pair of disjoint vertex sets $S,T\subseteq V(G)$ is \emph{$\E$-regular} if for every $A\subseteq S, B\subseteq T$ with $|A|\geq \E|S|, |B|\geq \E|T|$, we have $|d(A,B)-d(S,T)|<\E$.
	
	\begin{theorem}[Szemer\'edi's Regularity Lemma~\cite{AFKS}]\label{regularity}
		For every $m\in\N$ and $\E>0$, there exists an integer $S=S(m,\E)$ such that the following holds: Every graph $G$ on $n\geq S$ vertices has an equipartion $V(G)=V_1\cup \ldots\cup V_\ell$ and an induced subgraph $U$ with an equipartition $V(U)=U_1\cup \ldots\cup U_\ell$ such that: \begin{itemize}
			\item $m\leq \ell\leq S$.
			\item $U_i\subseteq V_i$ and $|U_i|=L\geq \lceil n/S\rceil$ for all $1\leq i\leq \ell$.
			\item All pairs $(U_i, U_j)$ with $i\neq j$ are $\E$-regular.
			\item All but at most $\E\binom{\ell}{2}$ of the pairs $1\leq i<j\leq\ell$ are such that $|d(V_i, V_j)-d(U_i, U_j)|<\E$.
		\end{itemize}
	\end{theorem}

	\subsection{\texorpdfstring{$k$}{2}-fold Sidon sets}
	
	The proof of Theorem~\ref{k3-k22} involves $k$-fold Sidon sets, which were first defined by Lazebnik and Verstra\" ete~\cite{LV}.
	
	Let $k\geq 1$ be an integer and $L(k)\coloneqq \{-k, -k+1, \ldots, 0, \ldots, k-1, k\}$. Let $n$ be relatively prime to all elements of $[k]$. For $i\in[4]$, let $c_i\in L(k)$ be such that $c_1+c_2+c_3+c_4=0$, but not all of them are $0$. Let $\mathcal{S}$ be the collection of sets $S\subseteq [4]$ such that $\sum_{i\in S} c_i=0$ and $c_i\neq 0$ for $i\in S$. A solution $(x_1, x_2, x_3, x_4)\in \mathbb{Z}_n^4$ to \begin{equation}\label{kfoldsidon}
		c_1x_1+c_2x_2+c_3x_3+c_4x_4=0
	\end{equation} is \emph{trivial} if there exists a partition of $[4]$ into sets $S, T\in \mathcal{S}$ such that $x_i=x_j$ for every $i,j \in S$ and every $i,j\in T$. Additionally, if $\mathcal{S}$ consists of only one set $S$, a solution $(x_1, x_2, x_3, x_4)\in \mathbb{Z}_n^4$ satisfying $x_i=x_j$ for every $i,j \in S$ is also trivial. A set $A\subseteq \mathbb{Z}_n$ is a \emph{$k$-fold Sidon set} if equation~\eqref{kfoldsidon} has only trivial solutions in $A$ for all $c_1,\ldots,c_4\in L(k)$ with $c_1+c_2+c_3+c_4=0$. The definition extends naturally to a set in $[n]$ where the condition that $n$ is relatively prime to all elements of $[k]$ may be dropped. Note that a $1$-fold Sidon set is precisely a Sidon set.
	
	Lazebnik and Verstra\" ete~\cite{LV} proved that for infinitely many $n$, there exists a $2$-fold Sidon set $A\subseteq \mathbb{Z}_n$ of size $|A|\geq \frac{1}{2}n^{1/2}-3$. They made the following general conjecture on $k$-fold Sidon sets for $k\geq 3$.
	
	\begin{conjecture}[Lazebnik and Verstra\" ete~\cite{LV}]\label{conj:sidon}
		For every integer $k\geq 3$, there exists a constant $c_k>0$ such that if $n\geq 1$, then there exists a $k$-fold Sidon set $A$ in $\mathbb{Z}_n$ (or in $[n]$) of size $|A|\geq c_k n^{1/2}$.
	\end{conjecture}
	
	Given a hypergraph $\mH$, a \emph{$t$-cycle} in $\mH$ is an alternating sequence $v_1, e_1, v_2, e_2, \ldots, v_t, e_t$ of distinct vertices $v_1,\ldots, v_t$ and distinct (hyper)edges $e_1,\ldots, e_t$ such that $v_i, v_{i+1}\in e_i$ for $i\in [t-1]$ and $v_t, v_1\in e_t$. Lazebnik and Verstra\" ete~\cite{LV} found the following connection between $k$-fold Sidon sets and $r$-partite $r$-uniform hypergraphs of girth at least five (i.e., a hypergraph without any $2$-cycle, $3$-cycle or $4$-cycle).
	
	\begin{theorem}[Proposition 2.4 in~\cite{LV}]\label{thm:::graph-sidon}
		Let $n, k, r$ be given positive integers and $n$ be odd. Suppose that there exists a Sidon set $B\subseteq [n]$ of size $r$ such that all differences of distinct elements in $B$ are relatively prime to $n$ and do not exceed $k$. If $A\subseteq \mathbb{Z}_n$ is a $k$-fold Sidon set, then there exists an $r$-partite $r$-uniform hypergraph of girth at least five on $rn$ vertices, with $|A|n$ edges.
	\end{theorem}
	
	Although Conjecture~\ref{conj:sidon} is still open, it was stated in~\cite{CT} that both Axenovich and Verstra\"ete observed that one can adapt Ruzsa's construction~\cite{Ruzsa} to construct $k$-fold Sidon sets $A\subseteq [n]$ (or $A\subseteq \mathbb{Z}_n$) of size $|A|\geq c_k n^{1/2}e^{-c_k\sqrt{\log n}}$ for every $k\geq 3$. Therefore, if $n$ is large, then we can choose an integer $k$ such that there exists a Sidon set $B\subseteq [n]$ with $|B|=5$ satisfying that all differences of distinct elements in $B$ are relatively prime to $n$ and do not exceed $k$. Taking $r=5$ and applying Theorem~\ref{thm:::graph-sidon}, we have the following corollary.
	
	\begin{cor}\label{prelim: hypergraph}
		There exists an $n$-vertex $5$-partite $5$-uniform hypergraph $\mH$ of girth at least five, with $\Omega(n^{3/2}e^{-c\sqrt{\log n}})$ edges for some constant $c>0$.
	\end{cor}
	
	\section{Proof of Theorem~\ref{clique-other}}\label{chi:::f>h}
	
	In this section, we estimate $g(n,H,F)$ when $\chi(F)>\chi(H)$, in particular, when $H$ is the complete graph $K_r$ and the graph $F$ that Constructor needs to avoid is an arbitrary graph with chromatic number larger than $r$.
	
	\begin{proof}[Proof of Theorem~\ref{clique-other}.]
		Let $\chi(F)=s\geq 3, r<s$ and $n$ be sufficiently large.
		
	\noindent\textbf{Lower bound:} In order for Constructor to create many copies of $K_r$ without creating $F$, Constructor fixes an equipartition $V_1,\ldots, V_{s-1}$ of $V(K_n)$ before the game starts, and she only claims edges between $V_i$ and $V_j$ for $1\leq i\neq j\leq s-1$: In particular Constructor will play disjoint games of Bipartite-JumbleG($\epsilon$) as Maker on each of the pairs $(V_i,V_j)$ with $\epsilon=2(s-1)(\log n/n)^{1/3}$. 
		
		As $\chi(F)=s$, Constructor's graph is $F$-free. Recall that Maker in the game of JumbleG($\epsilon$) tries to create a pseudorandom graph with edge density around $1/2$. By Observation~\ref{obs:::steal}, we have
		\[
		d(V_i,V_j)\geq \frac{1}{2}-\epsilon
		\]
		in her graph for every $1\leq i\neq j\leq s-1$. By standard counting arguments in pseudorandom graphs (see~\cite{BM} for more computational details), the number of copies of $K_r$ in Constructor's graph is at least 
		\[
		\binom{s-1}{r}\left(\frac{n}{s-1}\right)^r\left(\dfrac{1}{2}-\epsilon\right)^{e(K_r)}=\left(\frac{1}{2}\right)^{\binom{r}{2}}\binom{s-1}{r}\left(\frac{n}{s-1}\right)^r+o(n^r).
		\]
		
	\noindent\textbf{Upper bound:} The upper bound follows from a straightforward generalization of the upper bound for the Enforcer-Avoider game given in~\cite{BM}. As such, we will only provide a sketch of the proof. The main idea will be that Blocker will play JumbleG, and then under the contrary assumption that Constructor's graph contains many $K_r$'s, we will use Szemer\'edi's Regularity Lemma on Constructor's graph. We will count the number of copies of $K_r$ in the so-called cluster graph $R$. For a given graph $G'$, let $k_r(G')$ denote the number of copies of $K_r$ in $G'$. We will find that $k_r(R)>\ex(|V(R)|,K_r,K_s)$, which implies that the cluster graph contains a copy of $K_s$, and then by the Blow-up Lemma~\cite{KS}, Constructor's graph will contain $F$.
		
		Blocker will play as Maker in the game of JumbleG($\epsilon$), with $\epsilon= 2(\log n/n)^{1/3}$. Suppose for a contradiction that for some small constant $\delta>0$, Constructor has been able to build a graph $G$ with
		\[
		k_r(G)>\left(\frac{1}{2}\right)^{\binom{r}{2}}\binom{s-1}{r}\left(\frac{n}{s-1}\right)^r+\delta n^r.
		\]
		We will show that $F$ is a subgraph of $G$ as mentioned above, which is a contradiction as $G$ is $F$-free, completing the proof.
		
		Applying Szemer\'edi's Regularity Lemma to $G$ with parameters $\E_0$ and $m=\max\{r,\lceil 1/\E_0\rceil\}$, we get $S_0=S(m,\E_0)$ and equipartitions $V(G)=V_1\cup\ldots\cup V_{\ell_0}$ and $V(U)=U_1\cup\ldots\cup U_{\ell_0}$ according to Theorem~\ref{regularity}. We will choose the parameter $\E_0$ and take $n$ large enough such that 
		\[
		\epsilon=2(\log n/n)^{1/3}\ll \E_0 \ll \delta \text{\ \ \ and\ \ \ }\epsilon<1/S_0.
		\]
		
		Let $R$ be the auxiliary graph (cluster graph) whose vertex set consists of the clusters $U_1,\dots,U_{\ell_0}$ and $U_i\sim_R U_j$ if $d(U_i,U_j)\geq \delta+\epsilon$. Let $L_0\coloneqq|U_1|=\ldots=|U_{\ell_0}|$, and let $\tilde{U}$ be obtained from $U$ by deleting all edges inside $U_i$ for every $i\in[\ell_0]$. Note that any copy of $K_r$ in $G$ belongs to one of the following three types: 
		
		\begin{itemize}
			\item[(a)] Induced by the equipartition $V_1\cup \ldots \cup V_{\ell_0}$, i.e., $|V_i\cap V(K_r)|\leq 1$ for all $i\in[\ell_0]$; 
			
			\item[(b)] Inside some part $V_i$, i.e., $V(K_r)\subseteq V_i$ for some $i\in[\ell_0]$; 
			
			\item[(c)] $V(K_r)\not\subseteq V_i$ for all $i\in[\ell_0]$ and $|V_j\cap V(K_r)|\geq 1$ for some $j\in[\ell_0]$. 
		\end{itemize}
		It is then worth noticing that type (a) counts for the majority of $K_r$'s in $G$, and the number of $K_r$'s of type (a) is roughly the same as the number of $K_r$'s in $\tilde{U}$. By a standard graph counting argument, we obtain 
	\begin{align}
		k_r(\tilde{U})&\geq k_r(G)\left(\frac{\ell_0^rL_0^r}{n^r}-o(1)\right)\notag\\
		&\geq \left(\left(\frac{1}{2}\right)^{\binom{r}{2}}\binom{s-1}{r}\left(\frac{n}{s-1}\right)^r+\delta n^r\right)\left(\frac{\ell_0^rL_0^r}{n^r}-o(1)\right)\notag\\
		&\geq (1+o(1))\left(\left(\frac{1}{2}\right)^{\binom{r}{2}}\binom{s-1}{r}\left(\frac{\ell_0}{s-1}\right)^r+\delta\ell_0^r\right)L_0^r\label{krU}.
	\end{align}

		Now, let us upper bound $K_r(\tilde{U})$ in terms of the number of $K_r$'s in $R$. Since Blocker plays as Maker in the game of JumbleG($\epsilon$), Constructor cannot claim too many edges between a pair of disjoint vertex sets of large size, implying that
		\[
		d(U_i,U_j)\leq 1/2+\epsilon.
		\] 
		For each clique in $R$, the corresponding clusters in $\tilde{U}$ could correspond to many cliques in $\tilde{U}$, in particular, each clique could correspond to as many as
		\[
		\left(\frac{1}{2}+\epsilon+\E_0\right)^{\binom{r}{2}}L_0^r
		\]
		cliques in $\tilde{U}$. On the other hand, any set of $r$ vertices in $R$ which do not induce a clique must correspond to a collection of clusters, a pair of whom has density not exceeding $\delta+\epsilon$, and so these clusters span at most
		\[
		\left(\frac{1}{2}+\epsilon+\E_0\right)^{\binom{r}{2}-1}\left(\delta+\epsilon\right)L_0^r
		\]
		cliques in $\tilde{U}$. Therefore, 
		\[
		k_r(\tilde{U})\leq k_r(R)\left(\frac{1}{2}+\epsilon+\E_0\right)^{\binom{r}{2}}L_0^r+\left(\binom{\ell_0}{r}-k_r(R)\right)\left(\frac{1}{2}+\epsilon+\E_0\right)^{\binom{r}{2}-1}\left(\delta+\epsilon\right)L_0^r.
		\]
		Solving for $k_r(R)$ and using \eqref{krU}, we obtain

	\begin{align*}
	k_r(R)&\geq \frac{k_r(\tilde{U})-\binom{\ell_0}{r}\left(\frac{1}{2}+\epsilon+\E_0\right)^{\binom{r}{2}-1}\left(\delta+\epsilon\right)L_0^r}{\left(\frac{1}{2}+\epsilon+\E_0\right)^{\binom{r}{2}-1}\left(\frac{1}{2}+\E_0-\delta\right)L_0^r}\\
	&\geq\frac{(1+o(1))\left(\left(\frac{1}{2}\right)^{\binom{r}{2}}\binom{s-1}{r}\left(\frac{\ell_0}{s-1}\right)^r+\delta\ell_0^r\right)L_0^r-\binom{\ell_0}{r}\left(\frac{1}{2}+\epsilon+\E_0\right)^{\binom{r}{2}-1}\left(\delta+\epsilon\right)L_0^r}{\left(\frac{1}{2}+\epsilon+\E_0\right)^{\binom{r}{2}-1}\left(\frac{1}{2}+\E_0-\delta\right)L_0^r}\\
	&\geq\frac{(1+o(1))\left(\left(\frac{1}{2}\right)^{\binom{r}{2}}\binom{s-1}{r}\left(\frac{\ell_0}{s-1}\right)^r+\delta\ell_0^r\right)}{\left(\frac{1}{2}+\epsilon+\E_0\right)^{\binom{r}{2}-1}\left(\frac{1}{2}+\E_0-\delta\right)}-\frac{\binom{\ell_0}{r}\left(\delta+\epsilon\right)}{\left(\frac{1}{2}+\E_0-\delta\right)}\\
	&\geq\frac{(1+o(1))\left(\left(\frac{1}{2}\right)^{\binom{r}{2}}\binom{s-1}{r}\left(\frac{\ell_0}{s-1}\right)^r+\frac{\delta}{4}\ell_0^r\right)}{\left(\frac{1}{2}+\epsilon+\E_0\right)^{\binom{r}{2}-1}\left(\frac{1}{2}+\E_0-\delta\right)}>\binom{s-1}{r}\left(\left\lceil\frac{\ell_0}{s-1}\right\rceil\right)^r\geq \ex(\ell_0,K_r,K_s),
	\end{align*}
	 which implies that $R$ contains a copy of $K_s$. Therefore, $G$ contains a copy of $F$, and the proof is completed.
	\end{proof}
	
	\section{Proofs of Theorems~\ref{k3 and odd} and~\ref{odd-odd}}\label{chi:::f=h}
	
	\subsection{\texorpdfstring{$g(n,K_3,C_{2k+1})$}{2} for \texorpdfstring{$k\geq 2$}{2}}
	
	\begin{proof}[Proof of Theorem~\ref{k3 and odd} (i)] Let $k\geq 2$.
		
		\noindent\textbf{Upper bound. }By Observation~\ref{tri-obs} and Theorem~\ref{thm:::GTtri-cycle} (ii), 
		\[
		g(n,K_3,C_{2k+1})\leq \ex(n,K_3,C_{2k+1})=O(\ex(n,C_{2k})).
		\]
		
		\noindent\textbf{Lower bound. }We will give a strategy for Constructor such that Constructor can claim $\Omega(\ex(n,\C_{\leq 2k}))$ triangles in her graph without creating a $C_{2k+1}$. Let $S_k$ denote the star with $k$ edges. Recall that $\ex(n,\C_{\leq 2k})\leq \ex(n, C_{2k})=O(n^{1+1/k})$, and note that $\ex(n,\mathcal{C}_{\leq 2k})=\omega(n)$.
		
		The strategy consists of three phases. In Phase~\uppercase\expandafter{\romannumeral1}, Constructor builds many small stars.  In Phase~\uppercase\expandafter{\romannumeral2}, Constructor builds a bipartite graph with $\Theta(\ex(n,\C_{\leq 2k}))$ edges, with one part of this bipartite graph consisting of the centers of the stars built in the previous phase. In Phase~\uppercase\expandafter{\romannumeral3}, Constructor uses the small stars and the bipartite graph to create many triangles without creating any copies of $C_{2k+1}$.

		\textit{Phase~\uppercase\expandafter{\romannumeral1}:} Constructor greedily builds $n/4$ disjoint copies of $S_2$. This takes $t_1\coloneqq n/2$ rounds. Let $X$ be the set of all centers of these stars, and for each $x\in X$, let $S_x$ denote the set containing the two leaves of the star with center $x$. Let $\mathcal{S}\coloneqq X\cup\bigcup_{x\in X}S_x$ denote the set of all $3n/4$ vertices in any star, and let $Y\coloneqq V(K_n)\setminus \mathcal{S}$.
		
		\textit{Phase~\uppercase\expandafter{\romannumeral2}:} Let $G_0$ be a $\mathcal{C}_{\leq2k}$-free graph on $n/2$ vertices with $\ex(n/2,\mathcal{C}_{\leq2k})=\Theta(\ex(n,\mathcal{C}_{\leq2k}))$ many edges. By a standard argument, there exists a spanning bipartite graph $B_0\subseteq G_0$ with at least half of the edges, say with parts $X_0$ and $Y_0$, $|X_0|\geq |Y_0|$. Let $X_1\subseteq X_0$ be the $n/4$ highest degree vertices in $X_0$, and let $B\subseteq B_0$ be the subgraph of $B_0$ induced by $X_1\cup Y_0$. Since we delete at most half of the vertices of $X_0$, and the vertices we delete are the lowest degree vertices, 
		\[
		|E(B)|\geq \frac{|E(B_0)|}{2}\geq \frac{|E(G_0)|}{4}=\Theta(\ex(n,\mathcal{C}_{\leq 2k})).
		\]
		Let us consider $B$ embedded into the vertex set of $K_n$ such that $X_1=X$ and $Y_0\subseteq Y$. In Phase~\uppercase\expandafter{\romannumeral2}, Constructor will greedily claim as many edges of $B$ as possible, ending this phase once every edge of $B$ has been claimed. Let $t_2$ denote the number of rounds in Phase~\uppercase\expandafter{\romannumeral2}, or alternatively the number of edges of $B$ that Constructor claims, then
		\[
		t_2\geq \frac{|E(B)|-n/2}{2}=\Theta(\ex(n,\mathcal{C}_{\leq 2k})).
		\]
		
		\textit{Phase~\uppercase\expandafter{\romannumeral3}:} Now Constructor aims to build triangles without creating any copies of $C_{2k+1}$. For $x\in X, y\in Y$, we will say the pair $xy$ is \textit{good} if all the following hold:
		\begin{itemize}
			\item $xy$ was claimed by Constructor in Phase~\uppercase\expandafter{\romannumeral2},
			\item there exists at least one vertex $a\in S_x$ such that $ay$ is unclaimed, and
			\item none of the edges from $S_x$ to $y$ have been claimed by Constructor.
		\end{itemize}
		As long as there is at least one good edge $xy$ at the start of Constructor's turn, Constructor continues the game by claiming an edge between $S_x$ and $y$. Each such edge creates a triangle in Constructor's graph.
		
		\begin{claim}
			Constructor does not create a copy of $C_{2k+1}$ in Phase~\uppercase\expandafter{\romannumeral3}.
		\end{claim}
		
		\begin{proof}
			Suppose for a contradiction that Constructor creates a copy of $C_{2k+1}$, call it $C$. We will show that in this case, $B$ would need to contain some cycle of length at most $2k$. Note that every edge Constructor claims in Phase~\uppercase\expandafter{\romannumeral2} and Phase~\uppercase\expandafter{\romannumeral3} has one endpoint in $\mathcal{S}$, and the other in $Y$, and thus these edges induce a bipartite graph. This implies that $C$ contains at least one edge claimed in Phase~\uppercase\expandafter{\romannumeral1}, say the edge $xz$ for some $x\in X$ and $z\in S_x$. Since $z$ has degree $2$ in Constructor's graph, an edge $zy$ must have been claimed in Phase~\uppercase\expandafter{\romannumeral3} for some $y\in Y$. Thus, $xy$ was good at some point, in particular, $xy$ is in Constructor's graph. This implies that Constructor's graph contains a cycle $C'$ with $V(C')=V(C)\setminus\{z\}$. Repeating this argument for each edge $x'z'$ from Phase~\uppercase\expandafter{\romannumeral1} in $C$ yields a cycle $C^*$ with $|V(C^*)|\leq |V(C)|=2k+1$, and such that $C^*$ contains only vertices in $X\cup Y$. However, the only edges Constructor claimed in $X\cup Y$ are from $B$, so $C^*\subseteq B$, a contradiction.
		\end{proof}
		
		We claim that at the start of Phase~\uppercase\expandafter{\romannumeral3} there were $\Theta(\ex(n,\C_{\leq 2k}))$ good edges. Indeed, to see this, let an edge  $xy$ with $x\in X$, $y\in Y$ be called \emph{bad} if
		\begin{itemize}
			\item $xy$ was claimed by Constructor in Phase~\uppercase\expandafter{\romannumeral2}, and
			\item for every $a\in S_x$, the edge $ay$ was claimed by Blocker prior to the start of Phase~\uppercase\expandafter{\romannumeral3}.
		\end{itemize}
		Note that at the start of Phase~\uppercase\expandafter{\romannumeral3}, every edge that Constructor claimed in Phase~\uppercase\expandafter{\romannumeral2} is either good or bad. The number of bad edges is at most $\frac{t_1+t_2}{2}=(1+o(1))\frac{t_2}{2}$, and thus we have at least
		\[
		t_2-(1+o(1))\frac{t_2}{2}=(1+o(1))\frac{t_2}{2}=\Theta(\ex(n,\mathcal{C}_{\leq 2k}))
		\]
		good edges at the start of Phase~\uppercase\expandafter{\romannumeral3}. For at least half of these good edges, Constructor successfully claims a triangle, completing the proof.
	\end{proof}		
	
	\subsection{\texorpdfstring{$g(n,C_{2k+1},K_3)$}{2} for \texorpdfstring{$k\geq 2$}{2}}
	
	\begin{proof}[Proof of Theorem~\ref{k3 and odd} (ii)]
		\noindent \textbf{Lower bound. }Constructor will use a strategy similar to the one used in the proof of Theorem~\ref{clique-other}. Before the game starts, Constructor implicitly embeds a blow-up of $C_{2k+1}$ into $K_n$ where each part has size $n/(2k+1)$, i.e.,~Constructor fixes an equipartition $V_1,\ldots,V_{2k+1}$ of $V(K_n)$. Constructor will claim edges only in the blow-up, so her graph will be $K_3$-free. For each pair $(V_i,V_{i+1})$, Constructor claims the edges using a winning strategy of JumbleG($\epsilon$), with $\epsilon=2(2k+1)(\log n/n)^{1/3}$. Therefore, Constructor is able to claim at least \[\left(\frac{n}{2k+1}\right)^{2k+1}\left(\frac{1}{2}-2\epsilon\right)^{e(C_{2k+1})}=\left(1+o(1)\right)\left(\frac{n}{2(2k+1)}\right)^{2k+1}\] copies of $C_{2k+1}$.
		
		\noindent \textbf{Upper bound. }The proof is similar to that of Theorem~\ref{chi:::f>h} in Section~\ref{chi:::f>h}. Instead of finding $F$ in the cluster graph, we find triangles.
	\end{proof}

	\subsection{\texorpdfstring{$g(n, C_{2\ell+1},C_{2k+1})$}{2} for \texorpdfstring{$k>\ell\geq 2$}{2}}
	
	\begin{proof}[Proof of Theorem~\ref{odd-odd}] Let $k>\ell\geq 2$.
		
		\noindent \textbf{Upper bound.} By Observation~\ref{tri-obs} and Theorem~\ref{ref-oo}, 
		\[
		g(n,C_{2\ell+1},C_{2k+1})\leq \ex(n,C_{2\ell+1},C_{2k+1})=\Theta(n^\ell).
		\]
		
		\noindent \textbf{Lower bound.} We will use a construction from~\cite{LiorAsaf}. Let $G_{(2\ell+1)}$ be a \emph{half blow-up} of $C_{2\ell+1}$, that is, $V(G_{(2\ell+1)})=V_0\cup V_1\cup\ldots\cup V_{2\ell}$ where $|V_1|=|V_3|=\ldots=|V_{2\ell-1}|=\frac{n-\ell-1}{\ell}$ and $|V_0|=|V_2|=|V_4|=\ldots=|V_{2\ell}|=1$, and consecutive pairs $(V_i, V_{i+1})$ induce complete bipartite graphs (indices taken modulo $2\ell+1$). Notice that $G_{(2\ell+1)}$ is a $C_{2k+1}$-free graph with $\Theta(n^\ell)$ copies of $C_{2\ell+1}$. The strategy of Constructor in our game is to build a subgraph of $G_{(2\ell+1)}$ that still has $\Theta(n^\ell)$ copies of $C_{2\ell+1}$. During the game, call a vertex $v\in V(K_n)$ \emph{new} if no edges incident with $v$ have been claimed by either player. In the strategy we give, the players will play at most $n/5$ rounds in total, so the graph at any step will have at most $2n/5$ edges, thus we will always have new vertices at our disposal. 
		
		Constructor starts by choosing an arbitrary vertex $v_0\in V(K_n)$. We break our strategy into $\ell$ phases. In the $i$th phase, we will define a vertex $v_{2i}$, and a set $V_{2i-1}^*$, and Constructor will claim edges forming a complete bipartite graph between $\{v_{2(i-1)},v_{2i}\}$ and $V_{2i-1}^*$. For $1\leq i\leq \ell-1$, the $i$th phase proceeds as follows:
		
		Constructor creates a star with center $v_{2(i-1)}$ of size $\frac{2n}{15\ell}$ by connecting $v_{2(i-1)}$ to new vertices, and call the set of leaves of this star $L_i$. We then choose a new vertex to be $v_{2i}$, and in the next $\frac{n}{15\ell}$ rounds, Constructor claims edges from $v_{2i}$ to $L_i$. Let $V_{2i-1}^*\coloneqq N(v_{2(i-1)})\cap N(v_{2i})$, where the neighborhoods are taken in Constructor's graph.
		
		The $\ell$th phase proceeds similarly, with one small difference. Constructor creates a star of size $\frac{2n}{15\ell}$ with center $v_{2\ell-2}$, call the set of leaves $L_\ell$, and then choose a new vertex to be $v_{2\ell}$. Differing from previous phases, Constructor claims the edge $v_0v_{2\ell}$, and only then does she start claiming edges from $v_{2\ell}$ to $L_\ell$. Constructor claims a total of $\frac{n}{15\ell}-1$ edges from $v_{2\ell}$ to $L_{\ell}$, and again set $V_{2\ell-1}^*\coloneqq N(v_{2\ell-2})\cap N(v_{2\ell})$.
		
		In each of the $\ell$ phases, Constructor claims exactly $\frac{n}{5\ell}$ edges, so this strategy ends after $n/5$ total rounds. Note that the resulting graph is a subgraph of $G_{(2\ell+1)}$, thus is $C_{2k+1}$-free. Since $|V_{2i-1}^*|\geq \frac{n}{15\ell}-1$ for all $1\leq i\leq \ell$, our graph has at least
		\[
		\left(\frac{n}{15\ell}-1\right)^\ell=\Theta(n^\ell)
		\]
		copies of $C_{2\ell+1}$.
	\end{proof}
	
	\section{Proof of Theorem~\ref{k3-k22}}\label{chi:::f<h}
	
	
	
		
		\noindent\textbf{Upper bound. }By Observation~\ref{tri-obs} and Theorem~\ref{ref-k3-c4}, we have 
		\[
		g(n, C_3, C_4)\leq \ex(n, C_3, K_{2,2})=O(n^{3/2}).
		\]
		
		
		\noindent\textbf{Lower bound. }
		By Corollary~\ref{prelim: hypergraph}, there exists an $n$-vertex $5$-partite $5$-uniform hypergraph $\mH$ of girth at least five, with $\Omega(n^{3/2}e^{-c\sqrt{\log n}})$ edges for some constant $c>0$. Embed $V(\mH)$ into $V(K_n)$ arbitrarily, i.e.,~view $K_n$ and $\mH$ as having the same vertex set. Then, every hyperedge $e\in \mH$ induces a copy of $K_5$ in $K_n$, which we will denote by $K_5^{(e)}$. First, we prove that Constructor can claim a triangle in $K_5$.
		
		\begin{claim}\label{win5}
			\[
			g(5,C_3,C_4)\geq 1.
			\]
		\end{claim}
		
		\begin{proof}
			Recall that Constructor starts the game. By the second round, Constructor can claim two edges $xy$ and $yz$ such that $xz$ was not claimed by Blocker in Round 1. Then Blocker must claim $xz$ as their second edge. Note that there must be a vertex $w\in V(K_5)\setminus \{x,y,z\}$ such that all of the edges $wx$, $wy$ and $wz$ are unclaimed at the start of Constructor's third turn. Constructor can claim $wy$ in Round 3, and then one of $wx$ or $wz$ in Round 4, creating a triangle.
		\end{proof}
		
		Call the Constructor's strategy above of winning on five vertices \emph{StrOne}. We will use StrOne to give a strategy for Constructor such that she can build at least $\Omega(n^{3/2}e^{-c\sqrt{\log n}})$ triangles in her graph without creating any $C_4$. 
		
		As the game progresses, we sort the hyperedges $e\in E(\mathcal{H})$ into four types. We say $e$ is
		\begin{itemize}
			\item \textbf{won} if Constructor has claimed a triangle in $K_5^{(e)}$,
			\item \textbf{winning} if $e$ has not been won, but Constructor has claimed more edges than Blocker in $K_5^{(e)}$,
			\item \textbf{untouched} if no edges of $K_5^{(e)}$ have been claimed, and
			\item \textbf{lost} otherwise.
		\end{itemize}  
		Note that at every step of the game, these four types partition $E(\mathcal{H})$.
		
		In Round 1, Constructor arbitrarily chooses an untouched edge $e\in E(\mathcal{H})$, and claims an edge in $K_5^{(e)}$. In all following rounds, Constructor plays according to the following strategy:
		\begin{enumerate}
			\item[S1] If Blocker's last edge was played in a winning hyperedge $e$, Constructor plays an edge in $K_5^{(e)}$ according to StrOne.
			\item[S2] Otherwise, if there are unclaimed hyperedges, Constructor arbitrarily selects one and claims an edge in the corresponding $K_5$.
			\item[S3] If there are no unclaimed hyperedges and Blocker did not play in a winning hyperedge, Constructor arbitrarily chooses a winning hyperedge $e$ and plays as if Blocker had played an edge in $K_5^{(e)}$.
		\end{enumerate}
		Note that since $\mathcal{H}$ has girth $5$, and StrOne does not create any $C_4$'s inside any $K_5^{(e)}$, this strategy does not create a $C_4$ in Constructor's graph.
		As long as Constructor follows this strategy, eventually every hyperedge in $\mathcal{H}$ will end up either won or lost. We claim that at least half the hyperedges in $\mathcal{H}$ end up won. Indeed, S1 guarantees that every winning hyperedge eventually becomes won, thus the only way for a hyperedge to become lost is if Blocker claims an edge in an unclaimed hyperedge. S2 guarantees that for every lost hyperedge, a winning hyperedge is created (except possibly if Blocker claims an edge in the last untouched hyperedge, but since Constructor moves first, we still ends up with half the hyperedges winning or won).
		
		Thus, with this strategy, Constructor claims at least
		\[
		\frac{|E(\mathcal{H})|}{2}=\Omega(n^{3/2}e^{-c\sqrt{\log n}})
		\]
		triangles.
		
	
	\section{Acknowledgment}
	
	The authors are grateful to Bal\'azs Patk\'os for introducing the problem while visiting UIUC. The visit was partially funded by RTG DMS-1937241.

\end{document}